\documentclass[11pt]{amsart}
\usepackage{amsmath,amssymb,latexsym,soul, tikz}
\usepackage{cite,mathrsfs,amsfonts, verbatim}
\usepackage{color,enumitem,graphicx}
\usepackage{esint}
\usepackage[colorlinks=true,urlcolor=blue,
citecolor=red,linkcolor=blue,linktocpage,pdfpagelabels,
bookmarksnumbered,bookmarksopen]{hyperref}
\usepackage[english]{babel}

\usepackage[left=3cm,right=3cm,top=3cm,bottom=3cm]{geometry}

\usepackage[hyperpageref]{backref}
\usepackage{refcheck}
\usepackage{mathrsfs}
\usepackage{lmodern}
\numberwithin{equation}{section}

\newtheorem{theorem}{Theorem}[section]
\newtheorem{lemma}[theorem]{Lemma}
\newtheorem{example}[theorem]{Example}
\newtheorem{proposition}[theorem]{Proposition}
\newtheorem{corollary}[theorem]{Corollary}
\newtheorem{remark}[theorem]{Remark}

\theoremstyle{definition}

\newcommand{\R}{{\mathbb R}}

\newcommand{\eps}{\varepsilon}

\newcommand{\rO}{{\rm O}}

\def\XXint#1#2#3{{\setbox0=\hbox{$#1{#2#3}{\int}$ }
\vcenter{\hbox{$#2#3$ }}\kern-.6\wd0}}

\title[Existence and asymptotic behavior of nontrivial solutions to S-H]{Existence and asymptotic behavior of nontrivial solutions to the Swift-Hohenberg equation}

\author[G.\ Marino]{Greta Marino}
\address[G.\ Marino]{Dipartimento di Matematica e Informatica
\newline\indent
Universit\`a degli Studi di Catania
\newline\indent
Viale A. Doria 6 I-95125 Catania, Italy}
\email{greta.marino@dmi.unict.it}

\author[S.\ Mosconi]{Sunra Mosconi}
\address[S.\ Mosconi]{Dipartimento di Matematica e Informatica
\newline\indent
Universit\`a degli Studi di Catania
\newline\indent
Viale A. Doria 6 I-95125 Catania, Italy}
\email{mosconi@dmi.unict.it}

\subjclass[2010]{XXX}


\thanks{S. M. was partially supported by Gruppo Nazionale per l'Analisi Matematica, la Probabilit\`a e le loro Applicazioni (INdAM)}

\begin{document}

\begin{abstract}
In this paper, we discuss several results regarding existence, non-existence and asymptotic properties of solutions to $u''''+qu''+f(u)=0$, under various hypotheses on the parameter $q$ and on the potential $F(t)=\int_0^tf(s)\, ds$, generally assumed to be bounded from below. We prove a non-existence result in the case $q\le 0$ and an existence result of periodic solution for: 1) almost every suitably small (depending on $F$), positive values of $q$; 2) all suitably large (depending on $F$) values of $q$. Finally, we describe some conditions on $F$ which ensure that some (or all) solutions $u_q$ to the equation satisfy $\|u_q\|_\infty\to 0$, as $q\downarrow 0$.     
\end{abstract}

\maketitle

\begin{center}
	\begin{minipage}{12cm}
		\small
		\tableofcontents
	\end{minipage}
\end{center}

\medskip

\section{Introduction and main results}

We are interested in solutions to the ordinary differential equation
\begin{equation}
\label{SH}
u''''+qu''+F'(u)=0,
\end{equation}
where $F$ is a smooth (say, $C^2$) potential which we can freely assume to satisfy $F(0)=0$. For $q>0$, \eqref{SH} is the Swift-Hohenberg (briefly, S-H) equation while, for $q\leq 0$, it is known as the Extended Fisher-Kolmogorov (briefly, EFK) equation. Both equations have a large number of applications, which differ substantially according to various ranges of the parameter $q$. We refer to the monograph \cite{pt} for its various physical derivations and many qualitative results. Historically, \eqref{SH} was first considered for the double-well potential $F(t)=(1-t^2)^2$, due to its relevance in phase transition problems. Heteroclinic or homoclinic solutions have been obtained in \cite{pt, bs, m, smets} for S-H and in \cite{pt, kv, kkv} for EFK.  Roughly speaking, the EFK case has proven to be much more manageable than the S-H case, with some basic questions still left open for the latter.

After the seminal works of McKenna, Lazer and Walter (see \cite{mw, mw2, lm}), the study of \eqref{SH} for convex, coercive potentials became a major tool to understand the modelling of suspension bridges. Since then, much more refined, higher dimensional models have been developed, and a rather exhaustive exposition on the subject can be found in the monograph \cite{g}. In this paper we will focus on quasi-convex potentials $F$, i.e., those satisfying
\begin{equation}
\label{QC}
F'(t)t\geq 0,\qquad \forall t\in \R.
\end{equation}
Under this assumption, our main interest will be answering to the following questions, mainly motivated by some open problems described in \cite{lm2, ms}:
\begin{enumerate}
\item
Under what condition does \eqref{SH} possesses nontrivial (i.e., non-constant) global/bounded solutions?
\item
What is the behavior, as $q\to 0$ (i.e. transitioning from S-H to EFK), of such solutions?
\end{enumerate}

We now briefly describe what is known so far regarding the previous questions, along with the new results contained in the paper. Notice that some of our statements will involve quantities like
\[
\limsup_{t\to \pm\infty} \frac{F(t)}{t^2}, \qquad \liminf_{t\to \pm\infty} \frac{F(t)}{t^2}.
\]
Similar, but weaker, statements will hold under similar conditions on $f(t)=F'(t)$, namely involving the corresponding quantites
\[
\limsup_{t\to \pm\infty} \frac{f(t)}{t}, \qquad \liminf_{t\to \pm\infty} \frac{f(t)}{t}
\]
(which are more frequent in the literature), simply due to the inequalities
\[
\limsup_{t\to \pm\infty} \frac{F(t)}{t^2}\le \limsup_{t\to \pm\infty} \frac{f(t)}{t}, \qquad \liminf_{t\to \pm\infty} \frac{F(t)}{t^2}\ge \liminf_{t\to \pm\infty} \frac{f(t)}{t},
\]
which may be strict in some cases.

\subsection{The EFK case}

Let us first remark that, for coercive potentials, condition \eqref{QC} has proven to be almost equivalent to the absence of (nontrivial) bounded solutions to \eqref{SH}. Here is a more precise statement.

\begin{theorem}
[Theorems 3.1 and 5.1 in \cite{ms}]
Let $q\le 0$. If \eqref{QC} holds, the only bounded solutions to \eqref{SH} are constants. Moreover, in the class of coercive potentials with $\{F'(t)=0\}$ discrete, \eqref{QC} is actually {\em equivalent} to the existence of bounded nontrivial solutions.
\end{theorem}

Under assumption \eqref{QC}, one may still seek for global unbounded solutions. An immediate ODE argument shows that, if $F'$ is globally Lipschitz, all local solutions of \eqref{SH} are actually global (and thus, by the previous Theorem, unbounded).  However, the peculiar nature of \eqref{SH} allows the following {\em one-sided} generalization (notice that this holds {\em for any} $q\in \R$).

\begin{theorem}[Theorem 1 in \cite{bfgk}]
Let $q\in \R$ be arbitrary and $F\in C^2$ satisfy $F'(t)t>0$, for all $t\neq 0$. If 
\begin{equation}
\label{onesided}
\text{either}\quad \limsup_{t\to +\infty}\frac{F'(t)}{t}<+\infty \quad \text{or}\quad \limsup_{t\to -\infty}\frac{F'(t)}{t}<+\infty,
\end{equation}
then any solution to \eqref{SH} is globally defined.
\end{theorem}

Regarding non-existence, Gazzola and Karageorgis proved the following. Recall that the {\em Hamiltonian energy} of a solution $u$ of \eqref{SH} is
\begin{equation}
\label{he}
E_u(x):=\frac{q}{2}|u'(x)|^2+u'(x)u'''(x)+F(u(x))-\frac{|u''(x)|^2}{2},
\end{equation}
and testing \eqref{SH} with $u'$ shows that $E_u$ is constant for any such solution. Clearly, for local solutions of the ODE, $E_u$ can be  arbitrarily assigned through the initial conditions on $u$.

\begin{theorem}[Theorem 3 in \cite{gk}]
Let $q\le 0$. Suppose $F$ is a convex potential satisfying
\[
F(0)=0,\qquad F'(t)t\geq c|t|^{2+\eps} \quad \text{for $\eps>0$},\qquad F'(t)t\geq c F(t)\quad \forall \, |t|>>1
\]
for some $c>0$ and 
\[
\liminf_{|t|\to +\infty} \frac{F(\lambda t)}{F(t)^\alpha}>0
\]
for some $\lambda\in \ ]0, 1[$, $\alpha>0$. If $u$ solves \eqref{SH} in a neighborhood of $0$ and 
\begin{equation}
\label{conds}
\text{either}\quad u'(0)u''(0)-u(0)u'''(0)-qu(0)u'(0)\neq 0\qquad
 \text{or}\quad E_u\neq 0,
\end{equation}
then $u$ blows up in finite time.
\end{theorem}
Actually, the result proved in \cite{gk} gives much more informations on the location and behavior of the blow-up phenomenon. However, the previous Theorem immediately gives the following consequence, which we state here, due to its relevance in our framework.

\begin{corollary}
Let $q\le 0$. Suppose $F$ satisfies the assumptions of the previous theorem.
Then, the only globally defined solution to \eqref{SH} is $u\equiv 0$.
\end{corollary}

\begin{proof}
By translation invariance, $u(x+x_0)$ is a global solution to \eqref{SH}, for any $x_0\in \R$. Therefore, the first condition in \eqref{conds} must fail at any $x_0\in \R$, which implies that
\[
u'u''-uu'''-quu'\equiv 0.
\]
Deriving this relation, we obtain
\[
|u''|^2-u(u''''+qu'')-q|u'|^2=|u''|^2-q|u'|^2+F'(u)u\equiv 0
\]
and, being $F'(t)t\ge c|t|^{2+\eps}$ and $q\le 0$, we immediately deduce $u\equiv 0$.
\end{proof}

Using a classical technique due to Bernis \cite{b} and  \cite[Theorem 3.1]{ms}, we will remove most of the previous assumptions, proving the following result.
\begin{theorem}
\label{th1}
Let $q\le 0$, $ F \in C^2 $ be such that \eqref{QC} holds and 
\begin{equation} 
\label{liminf0}
\liminf_{|t| \to \infty} \frac{F'(t) }{t|t|^\eps}> 0, \qquad \eps>0.
\end{equation}
Then, the only globally defined solutions of \eqref{SH} are constants.
\end{theorem} 

It is worth noting that, while \eqref{onesided} and \eqref{liminf0} are roughly complementary conditions  to establish (or rule out) existence of global nontrivial solutions to \eqref{SH}, a truly necessary and sufficient condition is still missing. It is worth comparing with the well known Keller-Osserman necessary and sufficient condition 
\[
\int_{*}^{\pm\infty} |F(t)|^{-\frac 1 2}\, dt<+\infty
\]
for the blow-up of solutions to the second order ODE $u''-F'(u)=0$. Even for $q=0$, no such integral optimal condition is known for \eqref{SH}.

\subsection{The S-H case} As we will see, the situation for $q\ge 0$ is more complex.

Regarding non-existence of nontrivial solutions, the seemingly most up-to date results are the following.

\begin{theorem}[Theorem 1 in \cite{rtt}]
Let $F\in C^2$ satisfy 
\begin{equation}
\label{h0}
a|t|^{p+1}\le F'(t)t\le b|t|^{r+1}+ c|t|^{p+1},\quad \text{for some $a, b, c>0$ and $1\le r<p$ }.  
\end{equation}
Then, for any $q>0$\footnote{In \cite{rtt}, this theorem is actually proved for $0<q\le 2$, but a simple scaling argument shows its validity for any $q>0$. Indeed, if $u$ solves \eqref{SH}, then $u_\lambda(x):=u(x/\sqrt{\lambda})$ solves $u_\lambda''''+\lambda q u_\lambda''+F'_\lambda(u_\lambda)=0$, where $F_\lambda(t):=\lambda^2 F(t)$ satisfies \eqref{h0} with the same exponents and with constants $a, b, c$ multiplied by $\lambda^2$. If $\lambda\le 2/q$, one can then apply Theorem 1 in \cite{rtt}.}, there exists $E_0=E_0(a, b, c, p, r, q) \ge 0$ such that any solution to \eqref{SH}, satisfying $E_u>E_0$ (see \eqref{he} for the definition of $E_u$), blows up in finite time.

\end{theorem}

\begin{theorem}[Theorem 1 in \cite{fd}]
\label{fd}
Let $F\in C^2$ satisfy \eqref{h0}  and
\begin{equation}
\label{h1}
F''(t)>F''(0),\quad \text{for all $t\neq 0$}.
\end{equation}
If $q>0$ satisfies $q^2\le 4F''(0)$, the only globally defined solution to \eqref{SH} is $u\equiv 0$.
\end{theorem}

In order to investigate the optimality of the hypotheses in the previous result, we prove the following existence theorem. Notice that the main assumption is one-sided (much in the spirit of \eqref{onesided}) and can be required to hold at $+\infty$ instead.

\begin{theorem}
\label{th2}
Let $ F\in C^2$ satisfy  $0=F(0)=\min_\R F$ and 
\[
6\limsup_{t \to -\infty} \frac{F(t)}{t^2}< F''(0).
\]
For almost every $ q>0$ such that
\[
24\limsup_{t \to -\infty} \frac{F(t)}{t^2}< q^2<  4F''(0),
\]
 there exists a nontrivial periodic solution to \eqref{SH}.
\end{theorem}

The number $6$ in the statement is probably not optimal, however it allows the construction of an example showing that condition \eqref{h1} is essential for non-existence.

\begin{example}
We claim that there exists $F\in C^2$ such that $F''(0)>0$ and \eqref{h0} holds, with \eqref{SH} having a nontrivial periodic (thus {\em bounded}) solution, for almost every $q>0$ such that $q^2< 4 F''(0)$.

Indeed, one can easily construct $F\in C^2$ such that 
\begin{enumerate}
\item
$F''(0)>0=F(0)$, 
\item
$F'(t)t>0$, for all $t\neq 0$,
\item
it holds
\[
\limsup_{t \to -\infty} \frac{F(t)}{t^2}=0.
\]
\end{enumerate}
Then, Theorem \ref{th2} provides a periodic solution $u$ for almost every $q\in [0, 2\sqrt{F''(0)}]$. Using Taylor's formula,  \eqref{h0} holds with $1=r<p=2$ and some $a, b>0$, for $t$ in a {\em bounded} open neighborhood $U$ of $\overline{u(\R)}$. Moreover, we can modify $F$ outside of $\overline{u(\R)}$ so that \eqref{h0} holds anywhere. 
\end{example}

Still in \cite{fd}, the r\^{o}le of the condition $q<2\sqrt{F''(0)}$ is also discussed, through the following partial converse of Theorem \ref{fd}.

\begin{theorem}[Theorem 2 in \cite{fd}]
Suppose $F\in C^2$ is even,  satisfies \eqref{h0}, \eqref{h1} and the limit $ \lim_{t\to +\infty}\frac{F'(t)}{t^p}$ exists.
Then, for every $q>0$ such that $q^2>4F''(0)$, there exists a nontrivial periodic solution to \eqref{SH}.
\end{theorem}
Notice that, being $p>1$, \eqref{h0} forces
\[
 \lim_{t\to +\infty}\frac{F'(t)}{t^p}=B>0,
 \]
which implies the (much weaker) condition 
 \[
 \liminf_{|t|\to +\infty}\frac{F(t)}{t^2}=+\infty.
 \]
We generalize the previous Theorem as follows.

\begin{theorem}
\label{th3}
Let $F\in C^2$ satisfy  $0=F(0)=\min_\R F$. For any $q>0$ such that
\[
4F''(0)<q^2<2\liminf_{|t|\to +\infty}\frac{F(t)}{t^2},
\]
there exists a nontrivial periodic solution to \eqref{SH}.
\end{theorem}

\subsection{Asymptotic behavior}

In light of the previous discussion, under assumption \eqref{QC} it only makes sense to consider the asymptotic behavior of the solutions to \eqref{SH}, for $q\downarrow 0$. The starting point is a result proved by Lazer and McKenna.

\begin{theorem}[Theorem 2 in \cite{lm2}]
Let, for some $q_n\downarrow 0$, $\{u_n\}_n$ be a sequence of bounded nontrivial solutions to 
\[
u''''+q_nu''+ (1+u)_+-1=0,
\]
where $a_+=\max\{0, a\}$, for any $a\in \R$. Then, $\|u_n\|_\infty\to +\infty$.
\end{theorem}

We briefly say that {\em the nontrivial solutions to \eqref{SH}  are unbounded, as $q\downarrow 0$}, if the thesis of the previous theorem holds for any $q_n\downarrow 0$ and corresponding nontrivial solutions $\{u_n\}_n$ to \eqref{SH}.
The previous result has later been generalized as follows.

\begin{theorem}[Theorem 3.2 in \cite{ms}]
Let $F\in C^2$ satisfy $F''(0)>0$, ${\rm int}(\{F'=0\})=\emptyset$ and \eqref{QC}. Then, the nontrivial solutions to \eqref{SH} are unbounded, as $q\downarrow 0$.
\end{theorem}

The condition ${\rm int}(\{F'=0\})=\emptyset$ is readily seen to be necessary for the thesis (see \cite[Remark 3.2]{ms}). Here we focus on the necessity of the assumption $F''(0)>0$, proving the following.

\begin{theorem}
\label{th4b}
Let $F\in C^2$ satisfy  
\[
{\rm Argmin}(F)=\{0\}, \qquad \liminf_{|t|\to +\infty} \frac{F(t)}{t^2}>0,\qquad F''(0)=0.
\]
Then, for any $q>0$, there is a nontrivial periodic solution $u_q$ to \eqref{SH}, satisfying $\lim_{q\downarrow 0}\|u_{q}\|_\infty=0$.
\end{theorem}

The solutions found in the previous theorem minimize the corresponding variational energy (not to be confused with a Hamiltonian energy usually associated to \eqref{SH}). We conclude the paper showing that the vanishing (in the sense of the previous theorem) of nontrivial solutions to \eqref{SH} can involve {\em all the nontrivial solutions} rather than just the special ones found above. A scaling argument shows that this phenomenon occurs for homogeneous potentials.

\begin{theorem}
\label{th5b}
Let $r>2$. If $ \{u_n\}_n $ is any sequence of bounded solutions to 
\[
u''''+q_nu''+|u|^{r-2}u=0
\]
for some $q_n\downarrow 0$, then $\|u_n\|_\infty\to 0$.
\end{theorem}

\subsection{Structure of the paper}
In section 2 we will introduce the functional analytic setting we will use throughout the paper, together with some remarks on the minima of convex envelopes. In section 3 we will prove Theorem \ref{th1}. In section 4.1 we prove Theorem \ref{th2}, using a mountain pass procedure, together with Struwe's monotonicity trick to recover from the possible lack of the Palais-Smale condition on the relevant functional. In section 4.2 we apply the standard method of the Calculus of Variations to prove Theorem \ref{th3}. Section 5 is devoted to the study of the asymptotic behavior, as $q\downarrow 0$, of solutions to \eqref{SH}. In section 5.1 we prove Theorem \ref{th4b}, through a-priori estimates and a Jensen inequality involving convex envelopes. In section 5.2 we prove Theorem \ref{th5b}, through a scaling argument and some elementary Liouville-type techniques.

\section{Preliminary material}

Most of the results we will prove are based on critical point theory in suitable function spaces, which we will now describe. 

For any $ T> 0 $, we introduce the real Hilbert space
\[
H_T:=\bigl\{u: u\in H^2([0, T]), u'\in H^1_0([0, T])\bigr\},
\]
with scalar product
\[
\langle u, v\rangle_{H_T}=\int_0^Tu''v''\, dx+\int_0^T uv\, dx,
\]
and corresponding norm $\|u\|_{H_T}$.
For any measurable $v\colon [0, T]\to \R$  we will use the notations
\[
\|v\|_p=\|v\|_{L^p(0, T)}, \quad 1\le p\le +\infty,
\]
\[
\{|v|>K\}=\{x\in [0, T]: |v(x)|>K\}
\]
and similarly for other inequalities. We will still use  $\|v\|_\infty$ for $\|v\|_{L^\infty(\R)}$ when $v\colon \R\to \R$ and there is no risk of confusion. Given $F\in C^2(\R)$ and $q\in \R$, we define the functional $J_q\colon H_T\to \R$, henceforth called the (variational) {\em energy}, as
\[
J_q(u):=\int_{0}^T\frac{|u''|^2}{2}-q\frac{|u'|^2}{2}+F(u)\, dx,
\]
omitting the dependance on $T$.
Observe that we can freely add to $F$ and $J_q$ a constant so that $J_q(0)=F(0)=0$. 
Since we are mainly interested in potentials satisfying \eqref{QC}, notice that the latter inequality implies $F(0)=\min_{t\in \R} F(t)$. Thus, we can reduce to the case
\begin{equation}
\label{min}
0=F(0)=\min_{t\in \R} F(t),
\end{equation}
a weaker hypothesis we will sometimes assume. 

The link between solutions of \eqref{SH} and the functional $J_q$ is given in the following proposition.

\begin{proposition}[\cite{ms}]
\label{prop}
Let $u\colon [0, T]\to\R$ be a critical point for $J_q$ in $H_T$. Then, its even extension $\tilde{u}\colon [-T, T]\to \R$ defines a $2T$-periodic $C^4(\R)$ solution to \eqref{SH}.
\end{proposition}

We will furthermore use the following elementary observation.

\begin{lemma}
\label{lsc}
For any $T>0$, $J_q\colon H_T\to \R$ is $C^1$ and weakly sequentially lower semi-continuous.
\end{lemma}

\begin{proof}
The functional $J_q$ being $C^1$ immediately follows from the Sobolev embedding
\[
\|u\|_\infty\le C_T\|u\|_{H_T},
\]
which ensures that there is no need of a growth condition for $F$. To prove lower semi-continuity, let $ \{v_n\}_n \subseteq H _T$ be a sequence such that $ v_n \rightharpoonup v $, weakly in $ H_T$. In particular, $\{v_n\}_n$ is bounded in $H_T$, which implies boundedness in $C^{1,\alpha}([0, T])$, by Sobolev embedding. Therefore, $\{v_n\}_n$ is compact in $C^1([0, T])$ by Ascoli-Arzel\`a, which ensures, by Lebesgue domi\-nated convergence, that
\[
-q\int_0^T|v_n'|^2\, dx+\int_0^TF(v_n)\, dx \to -q\int_0^T|v'|^2\, dx+\int_0^TF(v)\, dx.
\]
Since the remaining term $\frac 1 2 \|v_n''\|_{2}^2$ is weakly sequentially lower semi-continuous by convexity, the thesis follows.
\end{proof}

We end this section with a general result on convex envelopes, which is of some interest in itself.
More precisely, given $ G\colon \mathbb R^N \to \mathbb R $, we let
\[
G^*(x):= \sup \{g(x): \text{$g$ is convex and $g(y) \le G(y)$, for all $y\in \R^N$} \}
\]
be the convex envelope of $G$.

\begin{lemma}
\label{argmin}
Suppose $G\colon \R^N\to \R$ is a lower semi-continuous function such that
\begin{equation}
\label{argminh}
\liminf_{|x|\to +\infty} \frac{G(x)}{|x|}\ge\alpha, \quad \text{for some $\alpha\in \ ]0, +\infty[$}.
\end{equation}
 Then,
\begin{equation}
\label{th}
{\rm Argmin}(G^*)={\rm co}\big({\rm Argmin}(G)).
\end{equation}
\end{lemma}

\begin{proof}
Clearly, ${\rm Argmin}(G)$, and thus ${\rm co}\big({\rm Argmin}(G))$, is compact and non empty, and we can suppose, without loss of generality, that $\min_{\R^N} G=0$. From \eqref{argminh} we can find $M>0$ such that $G(x)\geq \frac{\alpha}{2}|x|$, for any $|x|\geq M$, hence the function
\[
h(x)= \frac{\alpha}{2}(|x|-M)
\]
satisfies $h\leq G$ in $\R^N$. Since $G^*\geq h$ by construction, this implies that ${\rm Argmin}(G^*)$ is compact as well. 

Since $g\equiv 0$ satisfies $g\leq G$ and is convex, it holds $0\leq G^*\leq G$, which implies that ${\rm co}\big({\rm Argmin}(G)\big) \subseteq {\rm Argmin}(G^*)$.
We prove the opposite inequality by contradiction, and thus suppose that there exists $x_0$ such that 
\begin{equation}
\label{contr}
\text{$G^*(x_0)=0$ and $x_0\notin {\rm co}\big({\rm Argmin}(G)\big)=:C$.}
\end{equation}
By the Hanh-Banach Theorem, there exists $v\in \R^N$, $|v|=1$, such that
\begin{equation}
\label{jk}
\sup_{x\in C} \langle v, x\rangle=\langle v, x_1\rangle<\langle v, x_0\rangle,\quad \text{for some $x_1\in C$},
\end{equation}
where, by $\langle v, x\rangle$, we mean the standard duality coupling.
Let, for $\eps>0$ 
\[
g_\eps(x)=\eps\bigl\langle v, x-\frac{x_0+x_1}{2}\bigr\rangle, 
\]
and notice that for any $x\in C$ it holds, by \eqref{jk},
\[
g_\eps(x)\leq \eps\biggl(\sup_{y\in C}\langle v, y\rangle-\bigl\langle v, \frac{x_0+x_1}{2}\bigl\rangle\biggr) =\eps\bigl(\langle v, x_1\rangle -\bigl\langle v, \frac{x_0+x_1}{2}\bigr\rangle\biggr)=\eps\bigl\langle v, \frac{x_1-x_0}{2}\bigl\rangle<0.
\]
Therefore, for any $\eps>0$,
\begin{equation}
\label{gh}
\sup_{x\in C} g_\eps(x)<0.
\end{equation}
The set $\{ g_{\alpha/4}\geq h\}$ is compact, since 
\[
\limsup_{|x|\to +\infty}\frac{g_{\alpha/4}(x)}{h(x)}=\limsup_{|x|\to +\infty}\frac 1 2\frac{\bigl\langle v, x-\frac{x_0+x_1}{2}\bigr\rangle}{|x|-M}=\frac 1 2.
\]
Thus, $K:=\{g_{\alpha/4}\geq 0\}\cap \{g_{\alpha/4}\geq h\}$ is compact and, for any $x\in K$, it holds $G(x)>0$ since, otherwise, $x\in C$ and \eqref{gh} implies $g_{\alpha/4}(x)<0$, contradicting $x\in K$. Therefore, we can set
\[
\inf_{x\in K} G(x)=\beta>0.
\]
Now, for sufficiently small $\eps\in \ ]0, \alpha/4[$, it holds
\[
\sup_{x\in K}g_\eps(x)\leq \beta,
\]
and we claim that for such $\eps$ it holds $g_\eps\leq G$ in the whole $\R^N$. This is clearly true on $\{g_\eps\leq 0\}=\{g_{\alpha/4}\leq 0\}$, since $G\geq 0$. From the definition of $\eps$ it holds
\[
g_\eps(x)\leq \beta\leq G(x), \quad \text{for all $x\in K$}.
\]
Finally, on $\{g_{\alpha/4}\geq 0\}\cap \{g_{\alpha/4}< h\}$ one has 
\[
g_\eps(x)\leq g_{\alpha/4}(x)<h(x)\leq G(x),
\]
and the claim is proved. Therefore, being $g_\eps$ convex, we deduce $G^*\geq g_\eps$. By  \eqref{jk},
\[
G(x_0)\ge g_\eps(x_0)=\eps\bigl\langle v, \frac{x_0-x_1}{2}\bigr\rangle>0,
\]
which gives the desired contradiction to \eqref{contr}. 
\end{proof}

\begin{remark} 
Condition \eqref{argminh} is optimal in order to obtain \eqref{th}. Consider the $C^2(\R, \R)$ coercive function
\[
G(x)=
\begin{cases}
\log(1+x^2)&\text{if $x<0$},\\
x^2&\text{if $x\ge 0$}.
\end{cases}
\]
A straightforward computation shows that 
\[
G^*(x)=
\begin{cases}
0&\text{if $x<0$},\\
x^2&\text{if $x\ge 0$},
\end{cases}
\]
and \eqref{th} fails, since ${\rm Argmin}(G^*)=\ ]-\infty, 0]\neq \{0\}={\rm co}\big({\rm Argmin}(G)\big)$.
\end{remark}

\section{Non-existence of global solutions for superlinear EFK}

In this section we deal with the case $ q\le 0 $. The main result is the following. 

\begin{theorem}

Let $ F \in C^2 $ satisfy \eqref{QC} and 
\begin{equation} 
\label{liminf}
\liminf_{|t| \to \infty} \frac{F'(t) }{t|t|^\eps}> 0, \qquad \eps>0.
\end{equation}
Then, the only globally defined solutions of
\begin{equation}
\label{neg}
u''''+ qu''+ F'(u)= 0,
\end{equation}
for $ q \le 0 $ are constants.
\end{theorem}

\begin{proof}
We let, for simplicity, $p=-q\ge 0$ and $F'(t)=f(t)$.
The weak formulation of (\ref{neg}) is
\[
\int u'' \varphi''+pu' \varphi'+ f(u) \varphi\, dx= 0, \qquad \text{for all} \quad \varphi \in C_c^2(\R).
\]
Letting $ \varphi= u \eta $, we have
\[
 \int |u''|^2 \eta+ 2 u'' u' \eta'+ u'' u \eta''+p |u'|^2 \eta+ qu u' \eta'+ f(u) u \eta\, dx= 0 
\]
and, by Young's inequality,
\[
 \int |u''|^2 \eta+p |u'|^2 \eta+ f(u) u \eta \,dx\le \int \frac{1}{2} \eta |u''|^2+ \frac{1}{2} \frac{|\eta''|^2}{\eta} u^2+ \frac{p}{2} \eta |u'|^2+ \frac{p}{2} \frac{|\eta'|^2}{\eta} u^2 \,dx- 2\int u'' u' \eta'\, dx.
\]
It follows that
\begin{equation}
\label{int}
 \int |u''|^2 \eta+ p |u'|^2 \eta+ 2f(u)u \eta\, dx \le  \int u^2 \bigl(\frac{|\eta''|^2}{\eta}+ p\frac{|\eta'|^2}{\eta} \bigr)\,dx- 4 \int u'' u' \eta'\,dx.
\end{equation}
We estimate the last term integrating by parts as
\[
\int u'' u' \eta' \, dx = \int \bigl(\frac{|u'|^2}{2}\bigr)' \eta'\,  dx= -  \int \frac{|u'|^2}{2} \eta'' \,dx= -\frac{1}{2} \int u' (u' \eta'')\, dx= \frac{1}{2} \int u u'' \eta''+ u u' \eta''' \,dx.
\]
Moreover,
\[
\int u u' \eta''' dx= \int \bigl(\frac{u^2}{2}\bigr)' \eta''' dx= -\int \frac{u^2}{2} \eta '''' dx
\]
and, again by Young's inequality,
\[
\int u u'' \eta'' dx \le \int \frac{1}{2} \eta |u''|^2+ \frac{1}{2} \frac{|\eta''|^2}{\eta} u^2dx.
\]
Inserting into \eqref{int}, and using $f(t)t\ge 0$ for all $t\in \R$, we obtain
\begin{equation}
\label{int2}
 \int |u''|^2 \eta+ p |u'|^2 \eta+ f(u)u \eta \,dx \le C \int u^2 \biggl(\frac{|\eta''|^2+ p|\eta'|^2}{\eta}+ |\eta''''| \biggr) \, dx.
\end{equation}

Fix $ m \in \mathbb N $ large, $ R>1 $, and let $ \eta= \varphi_R^m $, where $ \varphi_R(x)= \varphi \bigl(\frac{x}{R}\bigr) $ and  $ \varphi \in C^\infty_c(\R, [0, 1])$ is a nonnegative cut-off function such that 
\[
 \varphi(x)= 
 \begin{cases} 
 1 & \text{if $|x|\leq 1$}, \\ 
 0 & \text{if $|x|\ge 2$}.
 \end{cases} 
 \]
Using $0\le \varphi_R\le 1$ and  $|\varphi_R^{(i)}|\leq C/R^i$, an explicit calculation shows that
\[
\frac{|\eta'|^2}{\eta} \le \frac{C}{R^2} \varphi_R^{m-2}\le \frac{C}{R^2}\varphi_R^{m-4},\qquad \frac{|\eta''|^2}{\eta} \le \frac{C}{R^4}\varphi_R^{m-4}, \qquad |\eta''''|\le \frac{C}{R^4}\varphi_R^{m-4}.
\]
With this choice, \eqref{int2} implies, through H\"older's inequality, $R>1$ and for any $m>4\frac{r}{r-2}$, $r>2$, that
\[
\begin{split}
\int f(u) u \varphi_R^m \,dx  &\le \frac{C}{R^2} \int u^2 \varphi_R^{m-4}\, dx\\
&\le \frac{C}{R^2}\int u^2\varphi_R^{\frac{2m}{r}}\varphi_R^{(1-\frac{2}{r})m-4}\, dx\\
&\le \frac{C}{R^2}\left(\int |u|^r\varphi_R^m\, dx\right)^{\frac{2}{r}}\left(\int \varphi_R^{m-4\frac{r}{r-2}}\, dx\right)^{1-\frac 2 r}\\
&\le \frac{C}{R^2}\left(\int |u|^r\varphi_R^m\, dx\right)^{\frac{2}{r}}R^{1-\frac 2 r}, 
\end{split}
\]
where we used that ${\rm supp}(\varphi_R)\subseteq [-2R, 2R]$.
Observe that (\ref{liminf}) implies that there exist $ K> 0 $, $\delta>0$ such that $ f(t) t> \delta |t|^{2+\eps} $, when $ |t|> K $. Letting $r=2+\eps$ and choosing $m>4\frac{2+\eps}{\eps}$, it follows, by Young's inequality, that
\[
\begin{split}
\delta \int_{\{|u|> K\}} |u|^r \varphi_R^m\, dx & \le \int_{\{|u|> K\}} f(u) u \varphi_R^m\, dx \le \frac{C}{R^{1+ \frac r 2}} \biggl(\int |u|^r \varphi_R^m\, dx \biggr)^{\frac 2 r}  \\
& \le \frac{C}{R^{1+ \frac 2 r}} \biggl[\biggl(\int_{\{|u|> K\}} |u|^r \varphi_R^m\, dx\biggr)^{\frac 2 r}+ \biggl(\int_{\{|u| \le K\}} |u|^r \varphi_R^m\, dx \biggr)^{\frac 2 r}\biggr] \\
& \le \frac{\delta}{2} \int_{\{|u|> K\}} |u|^r \varphi_R^m\, dx+  \frac{C_{\delta, r}}{(R^{1+ \frac 2 r})^{\frac{r}{r-2}}}+ \frac{CK^2}{R}.
\end{split}
\]  
Absorbing to the left the first term on the right, we obtain
\[
\int_{\{|u|> K\}}|u|^r\varphi_R^m\, dx\le \frac{C_{\delta, r}}{R^{\frac{r+2}{r-2}}}+ \frac{CK^2}{R}\to 0, \quad \text{for $R\to +\infty$},
\]
which implies that $\|u\|_\infty\le K$.  Therefore, by \cite[Theorem 3.1]{ms}, $ u $ is constant.

\end{proof}

\section{Periodic solutions for S-H}

\subsection{The one-sided sublinear case}

This section is devoted to the proof of the following theorem.

\begin{theorem}
\label{subex}
Let $ F\in C^2$ satisfy \eqref{min}. 
For almost every $ q>0$ such that
\[
24\limsup_{t \to -\infty} \frac{F(t)}{t^2}< q^2<  4F''(0),
\]
 there exists a nontrivial periodic solution to 
\eqref{SH}.
\end{theorem}

We will divide the proof in some lemmas, letting in the following
\[
\alpha:=\limsup_{t \to -\infty} \frac{F(t)}{t^2}.
\]

\begin{lemma}
\label{coer}
Let $F$ satisfy \eqref{min}. For any $0\le b< 2\sqrt{F''(0)}$ and any $T>0$, there exist $\eps>0$ and $\theta>0$ such that
\[
J_q(u) \ge \theta \|u\|_{H_T}^2, \qquad \text{for all $q\le b$ and $ \|u\|_{H_T}\le \eps$}.
\]

\end{lemma}

\begin{proof}
We will suppose that $F''(0)>0$, otherwise there is nothing to prove. We choose $\eta\in \ ] 0, 1[$ such that $b< 2\sqrt{\eta F''(0)}$. For a sufficiently small $M$ it holds, by Taylor's formula,
\[
F(t)\ge \eta F''(0) \frac{t^2}{2}, \quad \text{for all $|t|\le M$}.
\]
Since
\[
\|u\|_{\infty} \le C_T \|u\|_{H_T},
\]
we have that, for $\eps:= M/C_T$, it holds
\begin{equation}
\label{effe}
\int_0^TF(u)\, dx\ge \eta F''(0)\int_0^T \frac{u^2}{2}\, dx,\quad \text{ for all $\|u\|_{H_T}\le\eps$}.
\end{equation}
Integrating by parts, using $u\in H_T$ and applying Holder's inequality, we get
\begin{equation}
\label{osc}
\int_0^T |u'|^2 dx = \bigl[u u'\bigr]_0^T- \int_0^T u u'' dx  \le \biggl(\int_0^T u^2 dx \biggr)^{1/2} \biggl(\int_0^T |u''|^2 dx \biggr)^{1/2}.
\end{equation}
By Young's inequality in the form $2ab\le \lambda a^2+b^2/(4\lambda)$, we obtain
\[
\int_0^T|u'|^2\le \lambda\int_0^T|u''|^2\, dx+\frac{1}{4\lambda}\int_0^T u^2\, dx,
\]
which we rewrite as
\[
b\int_0^T\frac{|u'|^2}{2}\le \lambda b\int_0^T\frac{|u''|^2}{2}\, dx +\biggl(\frac{b}{4\lambda}-\eta F''(0)\biggr)\int_0^T u^2\, dx +\eta F''(0)\int_0^T\frac{u^2}{2}\, dx.
\]
Using \eqref{effe} we thus have, for all $\|u\|_{H_T}\le \eps$,
\[
b\int_0^T\frac{|u'|^2}{2}\le \lambda b\int_0^T\frac{|u''|^2}{2}\, dx+\int_0^T F(u)\, dx +\biggl(\frac{b}{4\lambda}-\eta F''(0)\biggr)\int_0^T u^2\, dx. 
\]
Rearranging and using $J_q(u)\ge J_b(u)$, we obtain
\[
J_q(u)\ge (1-\lambda b)\int_0^T|u''|^2\, dx+\biggl(\eta F''(0)- \frac{b}{4\lambda}\biggr)\int_0^Tu^2\, dx,
\]
if $\|u\|_{H_T}\le \eps$. Since $b^2<4\eta F''(0)$ by assumption, we can choose
\[
\bar\lambda\in \, \left]\frac{b}{4\eta F''(0)}, \frac 1 b\right[,
\]
obtaining the claim with
\[
\theta=\min\left\{ 1-\bar\lambda b, \eta F''(0)- \frac{1}{4\bar\lambda}\right\}>0.
\]

\end{proof}

\begin{lemma}
\label{negaz}
For any $ a>0$ such that $a^2>24\alpha$
and all $T>0$ such that
\begin{equation}
\label{condT}
\frac{a-\sqrt{a^2-24\alpha}}{2}< \frac{\pi^2}{T^2}< \frac{a+\sqrt{a^2-24\alpha}}{2},
\end{equation}
 there exists a function $ \tilde u \in H_T $ such that  $ J_q(\tilde u)< 0 $, for all $ q\ge a$. 

\end{lemma}

\begin{proof}

Fix $\theta>1$ and let, for  $ \mu\ge 2$, $  u_{\mu}(x):= \mu \big(\cos \bigl(\frac{\pi}{T} x \bigr)- \theta\big)\le (1-\theta)\mu $. An explicit calculation shows that
\begin{equation}
\label{calc}
\int_0^T |u''|^2dx=  \frac{\pi^4}{2T^3}, \qquad  \int_0^T |u'|^2 dx=  \frac{\pi^2}{2T},
\end{equation}
so that
\[
J_q( u_{\mu})= \mu^2 \frac{T}{4} \biggl(\frac{\pi}{T}\biggr)^2 \biggl[\biggl(\frac{\pi}{T}\biggr)^2- q\biggr]+ \int_0^T F( u_{\mu}) dx.
\]
Furthermore, for any $\eps>0$, there exists $ K> 0 $ such that 
\[
F(t)\le (\alpha+\eps) t^2, \qquad \text{for all} \quad t \le -K.
\]
Letting now $\mu\ge K/(\theta-1)$, we deduce that $ u_{\mu}\le (1-\theta)\mu \le -K $, and thus
\[
\int_0^TF(u_\mu)\, dx\le (\alpha+\eps)\mu^2\int_0^T\bigl(\cos(\frac{\pi}{T}x)-\theta\bigr)^2\, dx
\le (\alpha+\eps)\mu^2\bigl(\frac{T}{2}+\theta^2 T\bigr).
\]
Since $q\ge a$, we infer
\[
J_q(u_\mu)\le \mu^2\frac{T}{4}\biggl[\biggl(\frac{\pi}{T}\biggr)^4 -a\biggl(\frac{\pi}{T}\biggr)^2+2(\alpha+\eps)(1+2\theta^2)\biggr]
\]
and, letting $z=\pi^2/T^2$, $J_q(u_\mu)<0$ amounts to the existence of positive solutions to $z^2-az+2(\alpha+\eps)(1+2\theta^2)<0$, i.e.
\[
a^2<8(\alpha+\eps)(1+2\theta^2).
\]
Letting $\eps\to 0$ and $\theta\to 1$, we obtain the claim.
\end{proof}

Thanks to Lemmas \ref{coer} and \ref{negaz}, we obtain that $ J_q $ has the so-called \emph{mountain pass geometry}, uniformly on compact subintervals $[a, b]\subseteq \ ]\sqrt{24\alpha},  2\sqrt{F''(0)} [$. Indeed, it suffices to observe that $T=\sqrt{2}\pi/\sqrt{a}$ satisfies \eqref{condT} and consider $J_q$ on $H_T$. Moreover, $q\mapsto J_q(u)$ is monotone non-increasing in $q$, for any $u\in H_T$, and thus Struwe's monotonicity trick provides, for a.e. $q\in [a, b]$, a bounded Palais-Smale sequence $ \{u_n\}_n\subseteq H_T $ for $J_q$, at the mountain-pass level $c_{q, T}>0$.
Reasoning as in \cite{smets}, this in turn provides a mountain-pass critical point $ u_q\in H_T $, that is, through Proposition \ref{prop}, a $2T$-periodic solution of \eqref{SH}. The proof of Theorem \ref{subex} is thus complete. 

\begin{remark}
Inspecting the proof, one can rephrase the previous theorem putting more emphasis on the possible periods for which existence occurs. Namely, since we are assuming $6\alpha<F''(0)$ (otherwise there is nothing to prove), elementary calculus shows that then
\[
\inf_{z>0} z+\frac{6\alpha}{z}<2\sqrt{F''(0)}.
\]
Therefore, we can fix $T>0$ such that
\[
\frac{\pi^2}{T^2}+\frac{6\alpha T^2}{\pi^2}<2\sqrt{F''(0)}.
\]
Then, the previous proof shows that, for almost any $q$ such that
\[
\frac{\pi^2}{T^2}+\frac{6\alpha T^2}{\pi^2}<q<2\sqrt{F''(0)},
\]
there exists a $2T$-periodic solution to \eqref{SH}.
\end{remark}

\subsection{The superlinear case}

\begin{lemma}
Let $F\in C^2$ with $F(0)=F'(0)=0$. For any $q>2\sqrt{F''(0)} $, there exists $T>0$ such that $\inf_{H_T} J_q<0$.
\end{lemma}

\begin{proof}

 Choosing $ u(x)= \cos\bigl(\frac{\pi}{T}x\bigr) $, we have $\int_0^T u^2\, dx=\frac{T}{2}$.
For any $\eps>0$ there exists $M> 0 $ such that $ F(t) \le (F''(0)+\eps) t^2/2 $, for all $ |t| \le M $. For all $ 0< \lambda< M $, let now $ u_{\lambda}(x):= \lambda u(x) $. A direct calculation using \eqref{calc} shows that
\[
\begin{split}
J_q(u_\lambda) 
& \le \lambda^2 \frac{\pi^2}{4T} \biggl(\frac{\pi^2}{T^2}- q \biggr)+  \int_0^T F(u_{\lambda}) dx  \\
& \le \lambda^2 \frac{\pi^2}{4T} \biggl(\frac{\pi^2}{T^2}- q \biggr)+ \frac{F''(0)+\eps}{2}\lambda^2\int_0^T  u(x)^2 dx  \\
&\le \frac{\lambda^2}{4} \biggl[\frac{\pi^2}{T} \biggl(\frac{\pi^2}{T^2}- q \biggr)+ (F''(0)+\eps)T  \biggr]=:\lambda^2 A.
\end{split}
\]
Now, $A<0$ amounts to $z^2-qz+(F''(0)+\eps)<0$ having a positive solution $z=\pi^2/T^2$, which holds as long as $q^2-4(F''(0)+\eps)>0$. Letting $\eps\to 0$ completes the proof.

\end{proof}

\begin{remark}
\label{remT}
Clearly, the previous lemma also provides with a precise interval of possible periods $T$ for which $\inf_{H_T} J_q<0$. Indeed, the thesis holds for all $T$ such that
\[
\frac{q-\sqrt{q^2-4F''(0)}}{2}<\frac{\pi^2}{T^2}<\frac{q+\sqrt{q^2-4F''(0)}}{2}.
\]
Thus, in the degenerate case $F''(0)=0$, we see that, for all sufficiently large $T$ (precisely, for $qT^2>\pi^2$), the thesis of the previous lemma holds. This will be essential in the study of the asymptotic behavior, as $q\downarrow 0$, of solutions to the S-H equation.
\end{remark}

\begin{lemma}
\label{boundb}
Suppose that $F\in C^2$ satisfies $F\ge 0$ and
\[
\liminf_{|t|\to +\infty}\frac{F(t)}{t^2}\ge\alpha, \quad \text{for some $\alpha\in \ ]0, +\infty[$}.
\]
Then, $J_q$ is bounded from below on $H_T$, for any $q<\sqrt{2\alpha}$.
\end{lemma}

\begin{proof}
Clearly, we can suppose that $q>0$, otherwise $J_q\ge 0$ trivially. Since $J_q(u)\geq -\frac q 2 \|u'\|_{2}^2$, it suffices to bound from above $\|u'\|_{2}$, for any $ u \in H_T $ such that $ J_q(u)\leq 0 $. 
Therefore, we can suppose that $\|u'\|_2\neq 0$ and
\begin{equation}
\label{le0}
\int_0^T|u''|^2\, dx+2\int_0^TF(u)\, dx\leq q\int_0^T|u'|^2\, dx.
\end{equation}
In particular, it holds $\|u''\|_2^2\le q\|u'\|_2^2$ which, inserted into \eqref{osc}, gives
\[
\int_0^T|u'|^2\, dx\le  \biggl(\int_0^T u^2 dx \biggr)^{1/2}\biggl(q\int_0^T |u'|^2 dx \biggr)^{1/2},
\]
i.e., being $\|u'\|_2\neq 0$,
\[
\int_0^T |u'|^2 dx \le q\int_0^T u^2 dx.
\] 
For any $\theta\in\  ]0,1[$, let $M>0$ be such that $F(t)\ge \theta\alpha t^2$, for all $|t|>M$.
From the previous displayed inequality and \eqref{le0}, we obtain
\[
\begin{split}
\int_0^T |u'|^2 dx &\le q\biggl(\int_{\{|u|> M\}} u^2 dx + \int_{\{|u|\le M\}} u^2\, dx\biggr)\\
&\le q\biggl(\int_{\{|u|> M\}} \frac{F(u)}{\theta\alpha} dx + M^2|\{|u(x)|\le M\}|\biggr)\\
&\le q\biggl(\int_{0}^T \frac{F(u)}{\theta\alpha} dx + M^2T\biggr)\\
&\le q\biggl(\frac{q}{2\theta\alpha}\int_0^T |u'|^2 dx + M^2T\biggr),
\end{split}
\]
 which implies
\[
\biggl(1-\frac{q^2}{2\theta\alpha}\biggr)\int_0^T |u'|^2 dx\le qM^2T.
\]
If $q^2\le 2\theta\alpha$, then $1-\frac{q^2}{2\theta\alpha}>0$, and $\|u'\|_{2}$ is universally bounded. Being $\theta\in\  ]0, 1[$ arbitrary, we obtain the claim. 
\end{proof}

The direct method of Calculus of Variations immediately provides the following existence result. For a precise range of the values of $T$ for which the thesis holds, we refer to Remark \ref{remT}.

\begin{theorem}
\label{perex}
Let $F$ satisfy $\eqref{min}$. For any $q>0$ such that
\[
4F''(0)<q^2<2\liminf_{|t|\to +\infty}\frac{F(t)}{t^2},
\]
there exists a nontrivial periodic solution to \eqref{SH}.
\end{theorem}

\begin{proof}
By the previous two lemmas, $-\infty<\inf_{H_T} J_q<0$, for some $T>0$. Moreover, the proof of Lemma \ref{boundb} shows that there exists a constant $C=C(q, T, F)$ such that
\[
J_q(u)\leq 0\quad \Rightarrow\quad \|u'\|^2_{2}\leq C,
\]
which implies, for any $u\in \{u\in H_T:J_q(u)\leq 0\}$,
\[
\|u''\|_{2}^2\leq 2J_q(u)+q\|u'\|_{2}^2\leq qC.
\]
Finally, we can assume that 
\[
\liminf_{|t|\to +\infty}\frac{F(t)}{t^2}>0
\]
(otherwise there is nothing to prove), and thus there exists $M>0$ such that $F(t)\geq \alpha t^2/2$, for $|t|\geq M$. Still for $u$ such that $J_q(u)\leq 0$, it holds
\[
\begin{split}
\|u\|_{2}^2&\leq \int_{\{|u|\leq M\}} u^2\, dx +\frac{2}{\alpha}\int_{\{|u|\geq M\}} F(u)\, dx\\
&\le M^2T+\frac{2}{\alpha}\int_0^T F(u)\, dx\\
&\le M^2T+\frac{2}{\alpha}\biggl( J_q(u)+q\|u'\|_{2}^2\biggr)\\
&\le M^2T+\frac{2qC}{\alpha}.
\end{split}
\]
Hence, $\{u\in H_T:J_q(u)\leq 0\}$ is bounded, and thus weakly sequentially relatively compact. Now, Lemma \ref{lsc} provides the existence of a minimum $\bar u$, which is nontrivial, due to $J_q(\bar u)<0$. Finally, Proposition \ref{prop} implies that the $2T$-periodic even extension of $\bar u$ is a solution to \eqref{SH}.
\end{proof}

\section{Asymptotic behavior as $q\downarrow 0$}

\subsection{Periodic solutions of minimal energy}
In this section we discuss the asymptotic behavior of the periodic solutions of \eqref{SH} obtained in Theorem \ref{perex}, as $q\downarrow 0$. Clearly, in order to allow $q\downarrow 0$, we will assume in the following that $F''(0)=0$. 
Upon vertical and horizontal translations of the potential $F$, we can suppose that $0=F(0)=\min_{\R} F$.
We define
\begin{equation}
\label{defH}
\tilde F(t):= \min \{F(t), F(-t)\},\qquad H(t)=\Big( \widetilde F(\sqrt{|t|})\Big)^*.
\end{equation}

\begin{lemma}
\label{phi}
Suppose that $F\in C^2$ satisfies
\begin{equation}
\label{argmin}
{\rm Argmin}(F)=\{0\},
\end{equation}
\begin{equation}
\label{supq}
\liminf_{|t|\to +\infty}\frac{F(t)}{t^2}>0,
\end{equation}
\begin{equation}
\label{subq}
F''(0)=0.
\end{equation}
Then $H$, defined in \eqref{defH}, is an even convex function such that
\[
\varphi(t):=
\begin{cases}
 \frac{H(t)}{t} &\text{if $t\neq 0$},\\
 0&\text{if $t=0$}
 \end{cases}
 \]
is continuous and strictly increasing.

\end{lemma}

\begin{proof}
Since $G(t):=\widetilde{F}(\sqrt{|t|})$ is even, it follows immediately that $H$ is even. From \eqref{supq} we can find $\lambda>0$ such that $F(x)\geq \lambda |x|^2$, for any sufficiently large $|x|$, which implies that $\widetilde{F}(\sqrt{t})\geq \lambda |t|$, for sufficiently large $|t|$. Therefore, $G$ satisfies \eqref{argminh}, and thus Lemma \ref{argmin} provides ${\rm Argmin}(H)={\rm Argmin}(G)={\rm Argmin}(F)=\{0\}$, by assumption.

Moreover, by construction,
\[
H(t)\leq \widetilde{F}(\sqrt{|t|})\leq F(\sqrt{|t|}),
\]
so that \eqref{subq} implies 
\begin{equation}
\label{fi0}
\lim_{t\to 0}\frac{H(t)}{t}=0.
\end{equation}
In particular, $\varphi$ is continuous.
We then observe that, for $t\neq 0$,
\[
\varphi(t)=\frac{H(t)-H(0)}{t-0},
\]
and the convexity of $H$ implies that $\varphi$ is non-decreasing.  
To prove strict monotonicity, suppose, by contradiction, that $\varphi(t_1)=\varphi(t_2)$ for some $t_1<t_2$. Since $\varphi(t) t>0$ for $t\neq 0$, we can assume, without loss of generality, that $t_1>0$. Since $\varphi$ is non-decreasing, we infer $\varphi(t)=\varphi(t_1)=:\lambda>0$, for all $t\in [t_1, t_2]$, i.e. $H(t)=\lambda t$, for $t\in [t_1, t_2]$. Therefore, $\lambda t$ is a support line for $H$ and, by convexity, we obtain $H(t)\geq \lambda t$, for all $t>0$. This in turn implies that $H(t)\geq \lambda|t|$, being $H$ even, and we reach a contradiction, through \eqref{fi0}.
\end{proof}

\begin{theorem}
Suppose that  \eqref{argmin}, \eqref{supq}, \eqref{subq} hold and $ q T^2> \pi^2 $.  Then, the minimum $ u_{q, T} $ of $ J_q $ on $ H_{T} $ exists, is nontrivial and satisfies
\begin{equation}
 \label{claim}
 {\rm Osc}(u_{q, T})^2\leq qT^2\varphi^{-1}\biggl(\frac{q^2}{2}\biggr),
 \end{equation}
 where $\varphi$ is the function given in Lemma \ref{phi}.
 \end{theorem}

\begin{proof}
Clearly, \eqref{argmin} implies, eventually adding a constant, that \eqref{min} holds. Thus, Theorem \ref{perex} and Remark \ref{remT} show the existence part of the statement.
We let, for simplicity, $ u:= u_{q, T} $. From $\inf_{H_T} J_q\leq 0$, we deduce
\begin{equation}
\label{<0}
\int_0^T|u''|^2\, dx+2\int_0^TF(u)\, dx\leq q\int_0^T|u'|^2\, dx.
\end{equation}

Let $H$ be given by \eqref{defH}. By the previous lemma, we infer in particular that $\lim_{t\to +\infty}H(t)=+\infty$, so that $H$ is invertible on $[0, +\infty[$. By Jensen's inequality we have
\[
 H \biggl(\fint_0^T u^2 dx \biggr) \le \fint_0^T H(u^2) dx \le \fint_0^T F(u) dx, 
 \]
hence
\[
 \int_0^T u^2 dx \le T H^{-1} \biggl(\fint_0^T F(u) dx \biggr).
\]
Therefore, from \eqref{osc} and \eqref{<0} we have
\[
\int_0^T |u'|^2 dx  \le T^{1/2} \biggl[H^{-1} \biggl(\fint_0^T F(u) dx \biggr) \biggr]^{1/2} \biggl(q \int_0^T |u'|^2 dx \biggr)^{1/2}, 
\]
and simplifying we get
\[
 \int_0^T |u'|^2 dx \le q T H^{-1} \biggl(\fint_0^T F(u) dx \biggr). 
 \]
 Since $H^{-1}$ is increasing on $[0, +\infty[$, using again \eqref{<0} we have
\[
 \frac{1}{q T} \int_0^T |u'|^2 dx \le H^{-1} \biggl(\frac{q}{2} \fint_0^T |u'|^2 dx \biggr). 
 \]
 Letting $ z=z(q, T):= \frac{1}{q T} \int_0^T |u'|^2 dx $, and $\varphi(z)=H(z)/z$, the last inequality reads
\[
\varphi(z) \le \frac{q^2}{2}.
\] 
By Lemma \ref{phi}, $\varphi$ is invertible on $\{|\varphi|\leq q^2/2\}$, for sufficiently small $q$, so that
\begin{equation}
\label{u'2}
 \int_0^T |u'|^2 dx\leq qT\varphi^{-1}\biggl(\frac{q^2}{2}\biggr)
 \end{equation}
 and, using the standard inequality
 \[
 {\rm Osc}(u)^2\leq T\int_0^T|u'|^2\, dx,
 \]
 we obtain \eqref{claim}.

\end{proof}

\begin{corollary}
Suppose $F$ satisfies  \eqref{argmin}, \eqref{supq} and \eqref{subq}. Then, for any $q>0$, there exists $T(q)>0$ such that the minimum $u_{q, T(q)}$ of $J_q$ on $H_T$ exists, is nontrivial and satisfies
\[
\lim_{q\to 0}\|u_{q, T(q)}\|_\infty=0.
\]
\end{corollary}

\begin{proof}
For any $q>0$ we choose $T(q)$ such that 
\[
T(q)>\frac{\pi}{\sqrt{q}}, \qquad \lim_{q\to 0^+}qT^2(q)\varphi^{-1}\biggl(\frac{q^2}{2}\biggl)= 0
\]
(e.g., $\pi^2<qT^2(q)\leq K$, for some $K>\pi^2$, suffices). The previous Theorem provides the nontrivial solution $u_{q, T(q)}$ satisfying ${\rm Osc}(u_{q, T(q)})\to 0$ and, to complete the proof, it suffices to show that $u_{q, T(q)}(0)\to 0$. Suppose by contradiction that $u_{q_n, T(q_n)}(0)\geq \eps>0$, for some $q_n\to 0^+$, and let $u_n=u_{q_n, T(q_n)}$, $T_n= T(q_n) $. By \eqref{<0} and \eqref{u'2} it holds
\[
\int_0^{T_n}F(u_n)\, dx\leq \frac{q_n}{2}\int_0^{T_n}|u_n'|^2\, dx\leq \frac{1}{2}q_n^2T^2_n\varphi^{-1}\biggr(\frac{q_n^2}{2}\biggl)\to 0.
\]
Since  ${\rm Osc}(u_{n})\to 0$, for sufficiently large $n$ it holds $u_n(x)\geq u_n(0)-{\rm Osc}(u_n)>\eps/2$ for any $x$, which implies
\[
 \int_0^{T_n}F(u_n)\, dx\geq \frac{\eps}{2}T_n>\frac{\pi\eps}{2\sqrt{q_n}}\to +\infty,
 \]
 which contradicts the previous displayed estimate.

\end{proof}

\subsection{The case of homogeneous potentials}

Suppose now that $ F $ is homogeneous, e.g. $ F(u)= \frac{|u|^r}{r} $. We thus  focus on the model equation 
\begin{equation}
\label{u4}
u''''+ q u''+ |u|^{r-1}u= 0.
\end{equation}

\begin{theorem}

Let $r>2$. If $ \{u_n\} $ is a sequence of bounded solutions to \eqref{u4}, for some $q_n\downarrow 0$, then
\[
\lim_n\|u_n\|_\infty = 0.
\]

\end{theorem}

\begin{proof}

Suppose $u$ is any bounded nontrivial solution to \eqref{u4} and let $ v_{\lambda}(x):= u(\lambda x) $, for all $ \lambda> 0 $. Observe that $ \text{Osc}(v_{\lambda})= \text{Osc}(u) $ and that $ v_{\lambda} $ solves
\[
v_{\lambda}''''+ q \lambda^2 v_{\lambda}''+ \lambda^4 |v_\lambda|^{r-2}v_\lambda= 0.
\]
Furthermore, if
\[
w_{\lambda}(x):= \frac{v_{\lambda}(x)}{\text{Osc}(v_{\lambda})}= \frac{v_{\lambda}(x)}{\text{Osc}(u)},
\]
$ w_{\lambda} $ solves
\[
w_{\lambda}''''+ q \lambda^2 w_{\lambda}''+ \lambda^4 {\rm Osc}(u)^{r-2}|w_\lambda|^{r-2}w_\lambda= 0,
\]
for all $ \lambda> 0 $. Choosing $ \lambda^4= \text{Osc}(u)^{2-r} $, we obtain
\[
w''''+ \frac{q}{\text{Osc}(u)^{\gamma}} w''+ |w|^{r-2}w= 0.
\]
Applying this scaling argument to $u=u_n$, $q=q_n$ and letting $\rO_n={\rm Osc}(u_n)$, $w_n=u_n/\rO_n$, $\gamma=\frac r 2 -1>0$, we get
\begin{equation}
\label{w}
w_n''''+ \frac{q_n}{\rO_n^{\gamma}} w_n''+ |w_n|^{r-2}w_n= 0.
\end{equation}
We claim that 
\begin{equation}
\label{claim3}
\liminf_n\frac{q_n}{\rO_n^{\gamma}}\ge \beta>0.
\end{equation} Arguing by contradiction, suppose that, up to a not relabeled subsequence, 
\begin{equation}
\label{contrclaim3}
\lim_n\frac{q_n}{\rO_n^{\gamma}}= 0. 
\end{equation}
In particular, we can suppose, without loss of generality, that $ q_n/\rO_n^{\gamma}\leq 1 $. From \eqref{w} we have
\[
|w_n''''| \le |w_n''|+ |w_n|^{r-1}
\]
which, using  the $L^\infty$-interpolation inequality $\|u''\|_\infty\le C\sqrt{\|u\|_\infty\|u''''\|_\infty}$ and Young's inequa\-lity, implies
\begin{equation}
\label{c4}
 \|w_n\|_{C^4(\R)} \le C\|w_n\|_{\infty}.
\end{equation}

We now want to prove that $ \|w_n\|_{\infty} \le 1$.
  On the contrary, suppose that $ \|w_n\|_{\infty}> 1 $.  Then, without loss of generality, we can suppose that $ \sup w_n> 1 $ and, since $ \text{Osc}(w_n)= 1 $, we have $w_n(x)> 0$, for all $x\in \R$. This implies, through \eqref{w}, that 
\begin{equation}
\label{pp}
 \big(w_n''+ \frac{q_n}{\rO_n^\gamma} w_n\big)''=-|w_n|^{r-2}w_n< 0\quad \text{anywhere in $\R$},
 \end{equation}
  hence the function $ w_n''+ \frac{q_n}{\rO_n^{\gamma}} w_n $ is concave. Since it is also bounded, it must be constant, contradicting the strict inequality in \eqref{pp}.
  
Let $I_n$ be the interval $w_n(\R)$. From 
\[
|I_n|=1 ,\qquad I_n\subseteq [-1, 1],
\]
a standard compactness argument shows that, (up to a not relabeled subsequence), there exists an interval $J$ of length $1/2$ such that $ J\subseteq {\rm int}(I_n)$, for all $n$. Let $\lambda\in J\setminus \{0\}$. Being $\lambda\in w_n(\R)$, let $x_n$ be such that $w_n(x_n)=\lambda$, and let $v_n(x):=w_n(x+x_n)$. Clearly, $v_n$ solves \eqref{w} and satisfies ${\rm Osc }(v_n)\equiv 1$. Moreover, \eqref{c4} and $\|v_n\|_\infty=\|w_n\|_\infty\le 1$ show,  by  Ascoli-Arzel\`a's  Theorem,  that $\{v_n\}_n$ is a compact sequence in $ C_{\text{loc}}^3(\R) $, which we can suppose converges to some $v_0\in C^3_{\text{loc}}(\R)$.
Passing to the limit in the weak formulation of (\ref{w}) and using \eqref{contrclaim3}, we obtain 
\[
v_0''''+|v_0|^{r-2}v_0=0
\]
weakly, and thus strongly. Now, \cite[Theorem 3.1]{ms}  implies $v_0\equiv 0$, contradicting
\[
v_0(0)=\lim_nv_n(0)=\lim_nw_n(x_n)=\lambda\neq 0.
\]
Thus, \eqref{claim3} is proved, implying that, for any sufficiently large $n$, it holds 
\[
{\rm Osc}(u_n)^{\gamma}\leq \frac{2}{\beta}q_n,
\]
which proves that 
\begin{equation}
\label{O0}
{\rm Osc}(u_n)\to 0.
\end{equation}
It remains to prove that $u_n(0)\to 0$. The argument is the same as before and we only sketch it. If $u_n(0)\ge \eps>0$, then, by \eqref{O0}, for sufficiently large $n$ it holds $u_n(x)\ge \frac{\eps}{2}>0$, for all $x\in \R$, which implies
\[
\bigl(u_n''+q_n u_n\bigr)''=-|u|^{r-2}u<0 \quad \text{everywhere}.
\] 
Being $u_n''+q_nu_n$ bounded, it must be constant, contradicting the previous strict inequality. 
\end{proof}

\end{document}